\documentclass[12pt,a4paper]{article}

\usepackage{graphicx}%
\usepackage{amsmath,amsthm,amsfonts,amssymb} 
\usepackage{txfonts,eucal} 

\DeclareMathOperator\Aut{Aut}%
\DeclareMathOperator\Char{Char}%
\DeclareMathOperator\id{id}%
\DeclareMathOperator\GL{GL}%
\DeclareMathOperator{\GammaL}{\Gamma L}
\DeclareMathOperator\PGL{PGL}%
\DeclareMathOperator\PGammaL{P\Gamma L}%
\DeclareMathOperator\tr{tr}

\newcommand{\li}{\langle}
\newcommand{\re}{\rangle}
\newcommand{\lire}{\li\,\cdot\,,\cdot\,\re}
\newcommand{\quer}{\overline{(\cdot)}}
\newcommand{\ol}[1]{\overline{#1}}
\newcommand{\x}{\times} 

\newcommand{\rip}{\mathrel{\parallel_{r}}}
\newcommand{\lep}{\mathrel{\parallel_{\ell}}}

\newcommand{\riS}{\cS_{r}}
\newcommand{\leS}{\cS_{\ell}}

\newcommand{\Autle}{\ensuremath{\Gamma_{\ell}}}
\newcommand{\Autri}{\ensuremath{\Gamma_{r}}}
\newcommand{\Autleri}{\Autle}
\newcommand{\Autp}{\ensuremath{\Gamma_\parallel}}
\newcommand{\translations}{\ensuremath{\GL(H_H)}}
\newcommand{\inner}[1]{\ensuremath{\widetilde{#1}}}

\newcommand{\bPH}{\bP(H_F)} 
\newcommand{\dspH}{\bigl(\bPH,{\lep},{\rip}\bigr)}
\newcommand{\cLH}{\cL(H_F)} 

\newcommand{\overbracket}[1]{\ensuremath{\overset{\hbox{\rotatebox[origin=c]{-90}{[}}}{#1}}}

\renewcommand{\phi}{\varphi}

\newcommand{\cA}{{\mathcal A}} 

\newcommand{\cD}{{\mathcal D}}
\newcommand{\cE}{{\mathcal E}}
\newcommand{\cF}{{\mathcal F}}

\newcommand{\cL}{{\mathcal L}}

\newcommand{\cS}{{\mathcal S}}

\newcommand{\bF}{{\mathbb F}}

\newcommand{\bN}{{\mathbb N}}
\newcommand{\bP}{{\mathbb P}}

\newcommand{\bQ}{{\mathbb Q}}

\newtheorem{thm}{Theorem}[section]
\newtheorem{prop}[thm]{Proposition}
\newtheorem{cor}[thm]{Corollary}
\newtheorem{lem}[thm]{Lemma}
\theoremstyle{definition}

\theoremstyle{remark}
\newtheorem{rem}[thm]{Remark}
\newtheorem{exa}[thm]{Example}

\begin{document}

\author{Hans Havlicek \and Stefano Pasotti\thanks{This work was partially supported by GNSAGA of INdAM (Italy)} \and Silvia Pianta\footnotemark[1]\,\,\thanks{This work was partially supported by
Universit\`{a} Cattolica del Sacro Cuore (Milano, Italy) in the framework of Call D.1 2018}}
\title{Automorphisms of a Clifford-like parallelism}
\date{}

\maketitle

\begin{abstract}
In this paper we focus on the description of the automorphism group $\Autp$ of a Clifford-like parallelism $\parallel$ on a $3$-dimensional projective
double space $\dspH$ over a quaternion skew field $H$ (with centre a field $F$ of any characteristic).
We compare $\Autp$ with the automorphism group $\Autle$ of
the left parallelism $\lep$, which is strictly related to $\Aut(H)$. We build
up and discuss several examples showing that over certain quaternion skew
fields it is possible to choose $\parallel$ in such a way that $\Autp$ is
either properly contained in $\Autle$ or coincides with $\Autle$ even though
${\parallel}\neq{\lep}$.
\end{abstract}

\section{Introduction}

As a far-reaching generalisation of the situation in $3$-dimensional real
elliptic geometry, H.~Karzel, H.-J.~Kroll and K.~S\"{o}rensen coined the notion of
a \emph{projective double space}, that is, a projective space $\bP$ together
with a \emph{left parallelism} $\lep$ and a \emph{right parallelism} $\rip$ on
the line set of $\bP$ such that---loosely speaking---all ``mixed
parallelograms'' are closed \cite{kks-73}, \cite{kks-74}. It is common to
address the given parallelisms as the \emph{Clifford parallelisms} of the
projective double space. We shall not be concerned with the particular case
where ${\lep} = {\rip}$, which can only happen over a ground field of
characteristic two. All other projective double spaces are three-dimensional
and they can be obtained algebraically in terms of a quaternion skew field $H$
with centre $F$ by considering the projective space $\bP(H_F)$ on the vector
space $H$ over the field $F$ and defining $\lep$ and $\rip$ via left and
right multiplication in $H$. (See \cite{blunck+p+p-10a}, \cite{havl-15},
\cite{havl-16a}, \cite[pp.~75--76]{karz+k-88} and the references given there.)
In their work \cite{blunck+p+p-10a} about generalisations of Clifford
parallelism, A.~Blunck, S.~Pianta and S.~Pasotti pointed out that a projective
double space $\dspH$ may be equipped in a natural way with so-called
\emph{Clifford-like} parallelisms, namely parallelisms for which each equivalence class
is either a class of left parallel lines or a class of right
parallel lines. The exposition of this topic in \cite{havl+p+p-19a} serves as
major basis for this article.
\par
Our main objective is to describe the group of all collineations that preserve
a given Clifford-like parallelism $\parallel$ of a projective double space
$\dspH$. Since we work most of the time in terms of vector spaces, we shall
consider instead the underlying group $\Autp$ of all $\parallel$-preserving
semilinear transformations of the vector space $H_F$, which we call the
\emph{automorphism group} of the given parallelism. In a first step we focus on
the \emph{linear automorphisms} of $\parallel$. We establish in
Theorem~\ref{thm:cl-aut-lin} that the group of all these linear automorphism
does not depend on the choice of $\parallel$ among all Clifford-like
parallelisms of $\dspH$. Since $\lep$ and $\rip$ are also Clifford-like, it is
impossible to characterise Clifford parallelism in terms of its linear
automorphism group in our general setting of an arbitrary quaternion skew
field. On the other hand, there are projective double spaces in which there are
no Clifford-like parallelisms other than its Clifford parallelisms. This
happens, for instance, if $H$ is chosen to be the skew field of Hamilton's
quaternions over the real numbers. (It is worth noting that D.~Betten, R.~L\"{o}wen
and R.~Riesinger characterised Clifford parallelism among the topological
parallelisms of the $3$-dimensional real projective space by its (linear)
automorphism group in \cite{bett+l-17a}, \cite{bett+r-14a}, \cite{loew-17y},
\cite{loew-18z}, \cite{loew-17z}.) The next step is to consider the (full)
automorphism group $\Autp$. Here the situation is more intricate, since in
general the group depends on the underlying quaternion skew field as
well as the choice of $\parallel$. We know from previous work of S.~Pianta and
E.~Zizioli (see \cite{pian-87b} and \cite{pz90-coll}) that the left and right
Clifford parallelism of $\dspH$ share the same automorphism group, say
$\Autleri$. According to Corollary~\ref{cor:aut-not-ex}, $\Autleri$ cannot be a
proper subgroup of $\Autp$. In Section~\ref{sect:exa}, we construct a series of
examples showing that over certain quaternion skew fields it is possible to
choose $\parallel$ in such a way that $\Autp$ is either properly contained in
$\Autle$ or coincides with $\Autle$ even though ${\parallel}\neq{\lep,\rip}$.
\par
One open problem remains: Is there a projective double space $\dspH$ that
admits a Clifford-like parallelism $\parallel$ for which none of the groups
$\Autp$ and $\Autle$ is contained in the other one?

\section{Basic notions and results}\label{sect:basics}

Let $\bP$ be a projective space with line set $\cL$. We recall that a
\emph{parallelism} on $\bP$ is an equivalence relation on $\cL$ such that each
point of $\bP$ is incident with precisely one line from each equivalence class.
We usually denote a parallelism by the symbol $\parallel$. For each line
$M\in\cL$ we then write $\cS(M)$ for the equivalence class of $M$, which is
also addressed as the \emph{parallel class} of $M$. Any such parallel class is
a spread (of lines) of $\bP$, that is, a partition of the point set of $\bP$ by
lines. When dealing with several parallelisms at the same time we add some
subscript or superscript to the symbols $\parallel$ and $\cS$. The seminal book
\cite{john-10a} covers the literature about parallelisms up to the year 2010.
For the state of the art, various applications, connections with other areas of
geometry and historical remarks, we refer also to \cite{betta-16a},
\cite{bett+r-12a}, \cite{cogl-15a}, \cite{havl+r-17a}, \cite{karz+k-88},
\cite{loew-18z}, \cite{topa+z-18z} and the references therein.

\par
The following simple observation, which seems to be part of the folklore, will
be useful.
\par
\begin{lem}\label{lem:invar}
Let\/ $\bP$ and\/ $\bP'$ be projective spaces with
parallelisms\/ $\parallel$ and $\parallel'$, respectively. Suppose that
$\kappa$ is a collineation of\/ $\bP$ to\/ $\bP'$ such that any two $\parallel$-parallel
lines go over to $\parallel'$-parallel lines. Then $\kappa$ takes any\/ $\parallel$-class to a\/
$\parallel'$-class.
\end{lem}
\begin{proof} In $\bP'$, the $\kappa$-image of any $\parallel$-class is a
spread that is contained in a spread, namely some $\parallel'$-class. Any
proper subset of a spread fails to be a spread, whence the assertion follows.
\end{proof}
Let $H$ be a quaternion skew field with centre $F$; see, for example,
\cite[pp.~103--105]{draxl-83} or \cite[pp.~46--48]{tits+w-02a}. If $E$ is a
subfield of $H$ then $H$ is a left vector space and a right vector space over
$E$. These spaces are written as ${}_E H$ and $H_E$, respectively. We do not
distinguish between $_E H$ and $H_E$ whenever $E\subseteq F$. Given any $x\in H$ we denote by $\ol x$ the
\emph{conjugate quaternion} of $x$. Then $x=\ol x$ holds precisely when $x\in
F$. We write $\tr(x)=x+\ol x\in F$ for the \emph{trace} of $x$ and $N(x)=\ol x
x= x\ol x\in F$ for the \emph{norm} of $x$. We have the identity
\begin{equation}\label{eq:x1}
    x^2-\tr(x)x+N(x) = 0 .
\end{equation}
In $H_F$, the symmetric bilinear form associated to the quadratic form $N\colon
H\to F$ is
\begin{equation}\label{eq:<,>}
    \lire\colon H\x H\to F\colon (x,y)\mapsto \li x,y\re
    =\tr(x\ol{y})=x\ol{y}+y\ol{x} .
\end{equation}
\par
Let $\alpha$ be an automorphism of the quaternion skew field $H$. Then
$\alpha(F)=F$ and so $\alpha$ is a semilinear transformation of the
vector space $H_{F}$ with $\alpha_{|F}\colon F\to F$ being its accompanying
automorphism. Furthermore,
\begin{equation}\label{eq:tr+N}
    \forall\, x\in H\colon \tr\bigl(\alpha(x)\bigr)=\alpha\bigl(\tr(x)\bigr),\;
        N\bigl(\alpha(x)\bigr)=\alpha\bigl(N(x)\bigr), \;
        \ol{\alpha(x)}=\alpha(\ol{x}) .
\end{equation}
This is immediate for all $x\in F$, since here $\tr(x)=2x$, $N(x)=x^2$, and
$\ol x = x$. For all $x\in H\setminus F$ the equations in \eqref{eq:tr+N}
follow by applying $\alpha$ to \eqref{eq:x1} and by taking into account that
$\alpha(x^2)=\alpha(x)^2$ can be written in a unique way as an $F$-linear
combination of $\alpha(x)$ and $1$.
\par
The \emph{projective space} $\bPH$ is understood to be the set of all subspaces
of $H_F$ with \emph{incidence} being symmetrised inclusion. We adopt the usual
geometric terms: \emph{points}, \emph{lines} and \emph{planes} are the
subspaces of $H_F$ with vector dimension one, two, and three, respectively. We
write $\cLH$ for the line set of $\bPH$. The \emph{left parallelism} $\lep$ on
$\cLH$ is defined by letting $M_1\lep M_2$ precisely when there is a $g\in
H^*:=H\setminus\{0\}$ with $gM_1=M_2$. The \emph{right parallelism} $\rip$ is
defined in the same fashion via $M_1g=M_2$. Then $\dspH$ is a \emph{projective
double space} with $\lep$ and $\rip$ being its \emph{Clifford parallelisms}
(see \cite{kks-73}, \cite{kks-74}, \cite[pp.~75--76]{karz+k-88}).

A parallelism $\parallel$ of $\dspH$ is \emph{Clifford-like}, if each
$\parallel$-class is a left or a right parallel class
(see Def.~3.2. of \cite{havl+p+p-19a} where
 the construction of Clifford-like parallelisms appears frequently
in the more general framework of ``blending''; this point of view will be
disregarded here). Any Clifford-like parallelism $\parallel$ of $\dspH$ admits
the following explicit description:

\begin{thm}[{see \cite[Thm.~4.10]{havl+p+p-19a}}]\label{thm:cliff-like}
In $\dspH$, let $\cA(H_F)\subset\cLH$ denote the star of lines with centre
$F1$, let $\cF$ be any subset of $\cA(H_F)$, and define a relation $\parallel$ on $\cL(H_F)$
by taking the left parallel classes of all lines in $\cal F$ and the right
parallel classes of all lines in $\cA(H_F)\setminus\cF$.
This will be an equivalence relation (and hence, a parallelism) if, and only
if, the defining set $\cal F$ is invariant under the inner automorphisms of
$H$.
\end{thm}

We note that ---from an algebraic
point of view--- the lines from $\cA(H_F)$ are precisely the maximal subfields
of the quaternion skew field $H$.

\par

Let $\parallel$ be any parallelism on $\bPH$. We denote by $\Autp$ the set of
all mappings from $\GammaL(H_F)$ that act on $\bPH$ as $\parallel$-preserving
collineations. By
Lemma~\ref{lem:invar}, $\Autp$ is a subgroup of $\GammaL(H_F)$ and we shall
call it the \emph{automorphism group} of the parallelism $\parallel$. Even
though we are primarily interested in the group of all $\parallel$-preserving
collineations of $\bPH$, which is a subgroup of $\PGammaL(H_F)$, we investigate
instead the corresponding group $\Autp$. The straightforward task of rephrasing
our findings about $\Autp$ in projective terms is usually left to the reader.
\par
The Clifford parallelisms of the projective double space $\dspH$ give rise to
automorphism groups $\Gamma_{\lep}=:\Autle$ and $\Gamma_{\rip}=:\Autri$. We
recall from \cite[p.~166]{pian-87b} that
\begin{equation}\label{eq:le=ri}
     \Autle = \Autri.
\end{equation}
Equation~\eqref{eq:le=ri} is based on the following noteworthy geometric
result. In $\dspH$, the right (left) parallelism can be defined in terms of
incidence, non-incidence and left (right) parallelism. See, for example,
\cite[pp.~75--76]{karz+k-88} or make use of the (much more general) findings in
\cite[\S6]{herz-77a}, which are partly summarised in \cite{herz-77b} and
\cite{herz-80a}. In order to describe the group $\Autleri$ more explicitly, we
consider several other groups. First, the group of all \emph{left translations}
$\lambda_g\colon H\to H\colon x\mapsto g x$, $g\in H^*$, is precisely the group
$\translations$. The group $\translations$ is contained in $\GL(H_F)$ and it
acts regularly on $H^*$. Next, the automorphism group $\Aut(H)$ of the skew
field $H$ is a subgroup of $\GammaL(H_F)$. Finally, we write $\inner{H^*}$ for
the group of all inner automorphisms $\tilde{h} \colon H\to H \colon x\mapsto
h^{-1}x h$, $h\in H^*$, and so $\inner {H^*}$ is a subgroup of $\GL(H_F)$.
According to \cite[Thm.~1]{pian-87b} and \cite[Prop.~4.1 and
4.2]{pz90-coll}\footnote{We wish to note here that Prop.~4.3 of
\cite{pz90-coll} is not correct, since the group $\overbracket{K}$ from there
in general is not a subgroup of $\Aut(H)$.},
\begin{equation}\label{eq:semidir1}
    \Autleri = \translations \rtimes \Aut(H) = \GammaL(H_H).
\end{equation}
By symmetry of `left' and `right', \eqref{eq:semidir1} implies $\Autri =
\GL({}_H H) \rtimes \Aut(H)=\GammaL({}_H H)$, where $\GL({}_H H)$ is the group
of \emph{right translations}. Note that $\GammaL(H_H)=\GammaL({}_H H)$. From
this fact \eqref{eq:le=ri} follows once more and in an
algebraic way. By virtue of the Skolem-Noether
theorem \cite[Thm.~4.9]{jac-89}, the $F$-linear skew field automorphisms
of $H$ are precisely the inner automorphisms. We therefore obtain from \eqref{eq:semidir1} that
\begin{equation}\label{eq:semidir2}
    \Autleri\cap \GL(H_F) = \translations \rtimes \inner{H^*}.
\end{equation}
The subgroups of $\Autleri$ and $\Autleri\cap\GL(H_F)$ that stabilise $1\in H$
are the groups $\Aut(H)$ and $\inner{H^*}$, respectively.

\begin{rem}
The natural homomorphism $\GL(H_F)\to\PGL(H_F)$ sends the group from
\eqref{eq:semidir2} to the group of all $\lep$-preserving projective
collineations of $\bPH$. This collineation group can be written as the
\emph{direct product} of two (isomorphic) subgroups, namely the image of the
group of left translations $\GL(H_H)$ and the image of the group of right
translations $\GL({}_H H)$ under the natural homomorphism.
\end{rem}

\par
If $\alpha\colon H\to H$ is an antiautomorphism of the quaternion skew field
$H$, then $\alpha\in\GammaL(H_F)$ and $\alpha$ takes left (right) parallel
lines to right (left) parallel lines. In particular, the conjugation
$\quer\colon H\to H$ is an $F$-linear antiautomorphism of $H$. Therefore, the
set
\begin{equation}\label{eq:semilinswap}
    \bigl(\translations \rtimes \Aut(H)\bigr)\circ{\quer}
\end{equation}
comprises precisely those mappings in $\GammaL(H_F)$ that interchange the left
with the right Clifford parallelism. The analogous subset of $\GL(H_F)$ is
given by
\begin{equation*}\label{eq:linswap}
  \bigl(\translations \rtimes \inner{H^*}\bigr)\circ{\quer} .
\end{equation*}
Alternative proofs of the previous results can be found in
\cite[Sect.~4]{blunck+k+s+s-17z}.

\section{Automorphisms}

Throughout this section, we always assume $\parallel$ to be a Clifford-like
parallelism of $\dspH$ as described in Section~\ref{sect:basics}.
Our aim is to determine the group $\Autp$
of automorphisms of $\parallel$. In a first step we focus on the
transformations appearing in \eqref{eq:semidir1} and \eqref{eq:semilinswap}.

\begin{prop}\label{prop:par-preserv}
Let\/ $\parallel$ be a Clifford-like parallelism of\/ $\dspH$. Then the
following assertions hold.
\begin{enumerate}

\item\label{prop:par-preserv.a} An automorphism $\alpha\in\Aut(H)$
    preserves\/ $\parallel$ if, and only if, $\alpha(\cF)=\cF$.

\item\label{prop:par-preserv.b} An antiautomorphism $\alpha$ of the
    quaternion skew field $H$ preserves\/ $\parallel$ if, and only if,
    $\alpha(\cF)=\cA(H_F)\setminus\cF$.

\item\label{prop:par-preserv.c} For all $h\in H^*$, the inner automorphism
    $\tilde{h}$ preserves\/ $\parallel$.

\item\label{prop:par-preserv.d} For all $g\in H^*$, the left translation
    $\lambda_g$ preserves\/ $\parallel$.

\item\label{prop:par-preserv.e} If $\beta\in\GL(H_F)$ preserves $\lep$,
    then $\beta$ preserves also $\parallel$.
\end{enumerate}
\end{prop}

\begin{proof}
\eqref{prop:par-preserv.a} We read off from $\alpha(1)=1$ that
$\alpha\bigl(\cA(H_F)\bigr)=\cA(H_F)$ and from \eqref{eq:semidir1} that
$\alpha\in\Aut(H)\subset\Autleri$. The assertion now is an immediate
consequence of Theorem \ref{thm:cliff-like}.
\par
\eqref{prop:par-preserv.b} The proof follows the lines of
\eqref{prop:par-preserv.a} taking into account that $\alpha$ interchanges the
left with the right parallelism.
\par
\eqref{prop:par-preserv.c} \cite[Thm.~4.10]{havl+p+p-19a} establishes
$\tilde{h}(\cF)=\cF$. Applying \eqref{prop:par-preserv.a} we get
$\tilde{h}\in\Autp$.
\par
\eqref{prop:par-preserv.d} Choose any $\parallel$-class, say $\cS(L)$ with
$L\in\cA(H_F)$. In order to verify that $\lambda_g\bigl(\cS(L)\bigr)$ is also a
$\parallel$-class, we first observe that \eqref{eq:semidir1} gives
$\lambda_g\in\translations\subset\Autleri$. Next, we distinguish two cases. If
$L\in\cF$, then, by Theorem \ref{thm:cliff-like}, $\cS(L)=\leS(L)$ and so
$\lambda_g\bigl(\cS(L)\bigr)=\lambda_g\bigl(\leS(L)\bigr) =\leS(g
L)=\leS(L)=\cS(L)$.
If $L\in\cA(H_F)\setminus\cF$, then, by Theorem \ref{thm:cliff-like}, $\cS(L)=\riS(L)$.
Furthermore, \eqref{prop:par-preserv.c} gives $g L
g^{-1}\in\cA(H_F)\setminus\cF$. By virtue of these results and
\eqref{eq:le=ri}, we obtain
$\lambda_g\bigl(\cS(L)\bigr)=\lambda_g\bigl(\riS(L)\bigr) =\riS(g
L)=\riS( g L g^{-1})=\cS(g L g^{-1})$.
\par
\eqref{prop:par-preserv.e} By \eqref{eq:semidir2}, there exist $g,h\in H^*$ such that
$\beta=\lambda_g\circ\tilde{h}$. We established already in
\eqref{prop:par-preserv.d} and \eqref{prop:par-preserv.c} that
$\lambda_g,\tilde{h}\in\Autp$, which entails $\beta\in\Autp$.
\end{proof}

We proceed with a lemma that, apart from the quaternion formalism, follows
easily from \cite[Thm.~1.10, Thm~1.11]{luen-80a}; those theorems are about
spreads, their kernels and their corresponding translation planes. We follow
instead the idea of proof used in \cite[Thm.~4.3]{blunck+k+s+s-17z}.

\begin{lem}\label{lem:preauto}
Let $L\in\cA(H_F)$ and $\alpha\in\GammaL(H_F)$ be given such that $\alpha(1)=1$
and such that $\alpha$ takes one of the two parallel classes $\leS(L)$,
$\riS(L)$ to one of the two parallel classes $\leS\bigl(\alpha(L)\bigr)$,
$\riS\bigl(\alpha(L)\bigr)$. Then
\begin{equation}\label{eq:preauto1234}
    \forall\,x\in H,\;z\in L \colon
    \left\{\renewcommand\arraystretch{1.05}
    \begin{array}{l}
    \alpha(xz)= \left\{
                      \begin{array}{l@{\mbox{~~~if~~~}}l}
                      \alpha(x)\alpha(z) &  \alpha\bigl(\leS(L)\bigr)=\leS\bigl(\alpha(L)\bigr);\\
                      \alpha(z)\alpha(x) &  \alpha\bigl(\leS(L)\bigr)=\riS\bigl(\alpha(L)\bigr);
                      \end{array}
               \right.\\
    \alpha(zx)= \left\{
                      \begin{array}{l@{\mbox{~~~if~~~}}l}
                      \alpha(x)\alpha(z) &  \alpha\bigl(\riS(L)\bigr)=\leS\bigl(\alpha(L)\bigr);\\
                      \alpha(z)\alpha(x) &  \alpha\bigl(\riS(L)\bigr)=\riS\bigl(\alpha(L)\bigr).
                      \end{array}
                \right.
    \end{array}
    \right.
\end{equation}
\end{lem}
\begin{proof}
First, let us suppose that $\alpha$ takes the \emph{left} parallel class
$\leS(L)$ to the \emph{left} parallel class $\leS\bigl(\alpha(L)\bigr)$. We
consider $H$, on the one hand, as a
$2$-dimensional \emph{right} vector space $H_{L}$ and, on the other hand, as a $2$-dimensional \emph{right} vector space
$H_{\alpha(L)}$. By our assumption,
$\alpha$ takes $\leS(L)=\{g L \mid g\in H^*\}$ to
$\leS\bigl(\alpha(L)\bigr)=\{g'\alpha(L)\mid g'\in H^*\}$, \emph{i.e.}, the set
of one-dimensional subspaces of $H_{L }$ goes over to the set of
one-dimensional subspaces of $H_{\alpha(L)}$. Since $\alpha$ is additive, it is
a collineation of the affine plane on $H_{L }$ to the affine plane on
$H_{\alpha(L)}$. From $\alpha(0)=0$ and the Fundamental Theorem of Affine
Geometry, $\alpha$ is a semilinear transformation of $H_{L }$ to
$H_{\alpha(L)}$. Let $\phi_{L}\colon L \to \alpha(L)$ be its accompanying
isomorphism of fields. From $\alpha(1)=1$, we obtain
$\alpha(z)=\alpha(1z)=\alpha(1)\phi_{L}(z) =\phi_{L}(z)$ for all $z\in L$,
whence the $\phi_{L}$-semilinearity of $\alpha$ can be rewritten as
\begin{equation}\label{eq:preauto}
    \forall\,x\in H,\;z\in L \colon    \alpha(xz)=\alpha(x)\alpha(z) .
\end{equation}
\par
Next, suppose that $\alpha$ takes the \emph{left} parallel class $\leS(L)$ to
the \emph{right} parallel class $\riS\bigl(\alpha(L)\bigr)$. We proceed as
above except for $H_{\alpha(L)}$, which is replaced by the 2-dimensional
\emph{left} vector space $_{\alpha(L)}H$. In this way all products of
$\alpha$-images have to be rewritten in reverse order so that the equation in
\eqref{eq:preauto} changes to $\alpha(xz)=\alpha(z)\alpha(x)$.
\par
There remain the cases when $\alpha$ takes $\riS(L)$ to
$\leS\bigl(\alpha(L)\bigr)$ or $\riS\bigl(\alpha(L)\bigr)$. Accordingly, the
equation in \eqref{eq:preauto} takes the form $\alpha(zx)=\alpha(x)\alpha(z)$
or $\alpha(zx)=\alpha(z)\alpha(x)$.
\end{proof}

We now establish that any $\alpha\in\Autp$ fixing $1$ satisfies precisely one
of the two properties concerning $\alpha(\cF)$, as appearing in
Proposition~\ref{prop:par-preserv}~\eqref{prop:par-preserv.a} and
\eqref{prop:par-preserv.b}. Afterwards, we will be able to show that any such
$\alpha$ is actually an automorphism or antiautomorphism of the skew
field $H$.

\begin{prop}\label{prop:oneline}
Let $\alpha\in\Autp$ be such that $\alpha(1)=1$. If there exists a line
$L\in\cA(H_F)$ such that $\cS(L)$ and $\alpha\bigl(\cS(L)\bigr)$ are of the
same kind, that is, both are left or both are right parallel classes, then
$\alpha(\cF)=\cF$. Similarly, if there exists a line $L\in\cA(H_F)$ such that
$\cS(L)$ and $\alpha\bigl(\cS(L)\bigr)$ are of different kind, then
$\alpha(\cF)=\cA(H_F)\setminus\cF$.
\end{prop}
\begin{proof}
First, let us suppose that $\cS(L)=\leS(L)$ and
$\alpha\bigl(\cS(L)\bigr)=\leS\bigl(\alpha(L)\bigr)$. This means that $L$ and
$\alpha(L)$ are in $\cF$. We proceed by showing $\alpha(\cF)\subseteq\cF$. If
this were not the case, then a line $L'\in\cF$ would exist such that
$\alpha(L')\in\cA(H_F)\setminus\cF$, that is,
$\cS\bigl(\alpha(L')\bigr)=\riS\bigl(\alpha(L')\bigr)$. Furthermore, there
would exist quaternions $e\in L\setminus F$, $e'\in L'\setminus F$ and we would
have $e'e\neq e e'$. By Lemma~\ref{lem:preauto}, applied to $L$ and also to
$L'$, we would finally obtain $\alpha(e'e)=\alpha(e')\alpha(e)=\alpha(e e')$,
which is absurd due to $\alpha$ being injective. The same kind of reasoning can
be applied to $\alpha^{-1}\in\Autp$, whence $\alpha^{-1}(\cF)\subseteq \cF$.
Summing up, we have shown $\alpha(\cF)=\cF$ in our first case.
\par
The case when $\cS(L)=\riS(L)$ and
$\alpha\bigl(\cS(L)\bigr)=\riS\bigl(\alpha(L)\bigr)$ can be treated in an
analogous way and leads us to
$\alpha\bigl(\cA(H_F)\setminus\cF\bigr)=\cA(H_F)\setminus\cF$. Clearly, this is
equivalent to $\alpha(\cF)=\cF$.
\par
Let us now suppose that $\cS(L)$ and $\alpha\bigl(\cS(L)\bigr)$ are of
different kind, that is, one of them is a left and the other one is a right
parallel class. Then, by making the appropriate changes in the reasoning above,
we obtain $\alpha(\cF)=\cA(H_F)\setminus\cF$.
\end{proof}

On the basis of our previous results, we now establish our two main
theorems.

\begin{thm}\label{thm:cl-aut}
Let\/ $\parallel$ be a Clifford-like parallelism of\/ $\dspH$. Then a
semilinear transformation $\beta\in\GammaL(H_F)$ preserves\/ $\parallel$ if,
and only if, it can be written in the form
\begin{equation}\label{eq:cl-aut}
    \beta = \lambda_{\beta(1)}\circ \alpha,
\end{equation}
where $\lambda_{\beta(1)}$ denotes the left translation of $H$ by $\beta(1)$
and $\alpha$ either is an automorphism of the quaternion skew field $H$
satisfying $\alpha(\cF)=\cF$ or an antiautomorphism of $H$ satisfying
$\alpha(\cF)=\cA(H_F)\setminus\cF$.
\end{thm}
\begin{proof}
If $\beta$ can be factorised as in \eqref{eq:cl-aut}, then $\beta\in\Autp$
follows from Proposition~\ref{prop:par-preserv}~\eqref{prop:par-preserv.a},
\eqref{prop:par-preserv.b}, and \eqref{prop:par-preserv.d}.
\par
In order to verify the converse, we define
$\alpha:=\lambda_{\beta(1)}^{-1}\circ\beta$. Then $\alpha(1)=1$ and
$\alpha\in\Autp$ by Proposition
\ref{prop:par-preserv}~\eqref{prop:par-preserv.d}. We now distinguish two
cases.
\par
\emph{Case}~(i). There exists a line
$L\in\cA(H_F)$ such that $\cS(L)$ and $\alpha\bigl(\cS(L)\bigr)$ are of the
same kind. We claim that under these circumstances $\alpha\in\Aut(H)$.
\par
First, we confine ourselves to the subcase $\cS(L)=\leS(L)$. By the theorem of
Cartan-Brauer-Hua \cite[(13.17)]{lam-01a}, there is an $h\in H^*$ such that
$L':=h^{-1}Lh\neq L$. From
Proposition~\ref{prop:par-preserv}~{\eqref{prop:par-preserv.c}},
$\cS(L')=\tilde{h}\bigl(\cS(L)\bigr)$ is a left parallel class and, from
Proposition~\ref{prop:oneline}, the same holds for $\alpha\bigl(\cS(L')\bigr)$.
There exists an $e'\in L'\setminus L$ and, consequently, the
elements $1,e'$ constitute a basis of $H_{L}$. Given arbitrary quaternions
$x,y$ we may write $y=z_0+e' z_1$ with $z_0,z_1\in L$. By virtue of
Lemma~\ref{lem:preauto}, we obtain the intermediate result
\begin{equation}\label{eq:z-e'}
    \forall\, x\in H,\; z\in L \colon
    \alpha(xz)=\alpha(x)\alpha(z),\;
    \alpha(x e')=\alpha(x)\alpha(e').
\end{equation}
Using repeatedly the additivity of $\alpha$ and \eqref{eq:z-e'} gives
\begin{equation}\label{eq:auto}
\begin{aligned}
    \alpha(xy)& =\alpha(x z_0)        + \alpha\bigl((xe')z_1\bigr)  =\alpha(x)\alpha(z_0) + \alpha(xe')\alpha(z_1)\\
              & =\alpha(x)\bigl(\alpha(z_0) + \alpha(e')\alpha(z_1)\bigr)  =\alpha(x)\bigl(\alpha(z_0) + \alpha(e' z_1)\bigr) =\alpha(x)\alpha(y).
\end{aligned}
\end{equation}
Thus $\alpha$ is an automorphism of $H$.
\par
The subcase $\cS(L)=\riS(L)$ can be treated in an analogous way. It suffices to
replace $H_L$ with ${}_L H$ and to revert the order of the factors in all
products appearing in \eqref{eq:z-e'} and \eqref{eq:auto}.
\par
\emph{Case}~(ii). There exists a line $L\in\cA(H_F)$ such that $\cS(L)$ and
$\alpha\bigl(\cS(L)\bigr)$ are of different kind. Then, by reordering certain
factors appearing in Case~(i) in the appropriate way, the mapping $\alpha$
turns out to be an antiautomorphism of $H$.
\par
Altogether, since there exists a line in $\cA(H_F)$, $\alpha$ is an
automorphism or an antiautomorphism of $H$. Accordingly, from
Proposition~\ref{prop:par-preserv}~\eqref{prop:par-preserv.a} or
\eqref{prop:par-preserv.b}, $\alpha(\cF)=\cF$ or
$\alpha(\cF)=\cA(H_F)\setminus\cF$.
\end{proof}

\begin{thm}\label{thm:cl-aut-lin}
Let\/ $\parallel$ be a Clifford-like parallelism of\/ $\dspH$. Then the group\/
${\Autp}\cap\GL(H_F)$ of linear transformations preserving\/ $\parallel$
coincides with the group\/ ${\Autle}\cap\GL(H_F)$ of linear transformations
preserving the left Clifford parallelism\/ $\lep$.
\end{thm}
\begin{proof}
In view of Proposition~\ref{prop:par-preserv}~\eqref{prop:par-preserv.e} it
remains to show that any $\beta\in{\Autp}\cap\GL(H_F)$ is contained in
${\Autle}\cap\GL(H_F)$. From \eqref{eq:cl-aut}, we deduce
$\beta=\lambda_{\beta(1)}\circ\alpha$, where $\alpha\in\GL(H_F)$ is an
automorphism of $H$ such that $\alpha(\cF)=\cF$ or an antiautomorphism of $H$
such that $\alpha(\cF)=\cA(H_F)\setminus\cF$. There are two possibilities.
\par
\emph{Case}~(i). $\alpha$ is an automorphism. By the
Skolem-Noether theorem, $\alpha$ is inner. Consequently, \eqref{eq:le=ri} and
\eqref{eq:semidir2} give
$\beta\in{\Autleri}\cap\GL(H_F)$.
\par
\emph{Case}~(ii). $\alpha$ is an antiautomorphism. Again by Skolem-Noether, the
product $\alpha':=\alpha\circ\quer$ of the given $\alpha$ and the conjugation
is in $\inner{H^*}$. The conjugation fixes $1$ and sends any $x\in H$ to
$\ol{x}=\tr(x) -x\in F1+Fx$. Therefore, all lines of the star $\cA(H_F)$ remain
fixed under conjugation. The inner automorphism $\alpha'$ fixes $\cF$ as a set
\cite[Thm.~4.10]{havl+p+p-19a}. This gives $\alpha(\cF)=\cF$ and contradicts
$\alpha(\cF)=\cA(H_F)\setminus\cF$. So, the second case does not occur.
\end{proof}

Theorem~\ref{thm:cl-aut-lin} may be rephrased in the language of projective
geometry as follows: if a \emph{projective collineation} of $\bPH$ preserves a
\emph{single} Clifford-like parallelism $\parallel$ of $\dspH$, then \emph{all}
Clifford-like parallelisms of $\dspH$ (including $\lep$ and $\rip$) are
preserved. This means that a characterisation of the Clifford-parallelisms of
$\dspH$ by their common group of linear automorphisms \big(or by the
corresponding subgroup of the projective group $\PGL(H_F)$\big) is out of reach
whenever there exist Clifford-like parallelisms of $\dspH$ other than $\lep$
and $\rip$. (Cf.\ the beginning of Section~\ref{sect:exa}.) Indeed, by
\cite[Thm.~4.15]{havl+p+p-19a}, any Clifford-like parallelism of this kind is
not Clifford with respect to any projective double space structure on $\bPH$.

\begin{cor}\label{cor:skolem-noether-pasotti} Let $\alpha_1\in\Autp$ be a
fixed automorphism of $H$. Then the following assertions hold.
\begin{enumerate}
\item\label{cor:skolem-noether-pasotti.a} All automorphisms
$\alpha$ of the skew field $H$ satisfying
    $(\alpha_1)_{|F}=\alpha_{|F}$ are in the group\/ $\Autp$.
\item\label{cor:skolem-noether-pasotti.b} All antiautomorphisms
    $\alpha$ of the skew field $H$ satisfying
    $(\alpha_1)_{|F}=\alpha_{|F}$ are not in the group\/ $\Autp$.
\end{enumerate}
The whole statement remains true if the words ``automorphism'' and
``antiautomorphism'' are switched.
\end{cor}

\begin{proof}
\eqref{cor:skolem-noether-pasotti.a} By the Skolem-Noether theorem,
$\alpha^{-1}\circ\alpha_1$ is an inner automorphism of $H$. Thus, from
Proposition~\ref{prop:par-preserv}~\eqref{prop:par-preserv.c},
$\alpha^{-1}\circ\alpha_1\in\Autp$, which implies $\alpha\in\Autp$.
\par
\eqref{cor:skolem-noether-pasotti.b} The conjugation $\quer$ is $F$-linear. We
therefore can apply \eqref{cor:skolem-noether-pasotti.a} to $\alpha\circ\quer$
and in this way we obtain $\alpha\circ\quer\in \Autp$. The proof of
Theorem~\ref{thm:cl-aut-lin}, Case~(ii), gives $\quer\notin\Autp$. Hence
$\alpha\notin\Autp$ as well.
\end{proof}

Theorem~\ref{thm:cl-aut} and Corollary~\ref{cor:skolem-noether-pasotti} (with
$\alpha_1:=\id$) together entail that
\begin{equation*}
    \bigl\{\alpha\in\Autp\mid\alpha(1)=1\bigr\}\subset\Aut(H)\circ\bigl\{\id_H,\quer\bigr\}.
\end{equation*}
In particular, for all $h\in H^*$, the inner automorphism $\tilde{h}$ is in
$\Autp$, whereas the antiautomorphism $\tilde h\circ\quer$ of the skew field
$H$ does not belong to $\Autp$.
\par
Theorem~\ref{thm:cl-aut-lin} motivates to compare the automorphism
groups $\Autp$ and $\Autle$ with respect to inclusion. This leads to four
(mutually exclusive) possibilities as follows:
\begin{gather}
    \Autp = \Autle ,\label{eq:aut-equal}\\
    \Autp\subset\Autle , \label{eq:aut-proper}\\
    \Autp\supset\Autle , \label{eq:aut-not-ex}\\
    \Autp\not\subseteq\Autle\mbox{~a}\mbox{nd~}\Autp\not\supseteq\Autle .
    \label{eq:aut-open}
\end{gather}
In Section~\ref{sect:exa}, it will be shown, by giving illustrative examples,
that each of \eqref{eq:aut-equal} and \eqref{eq:aut-proper} is satisfied by
some Clifford-like parallelisms. The situation in \eqref{eq:aut-not-ex} does
not occur due to Corollary~\ref{cor:aut-not-ex} below. Whether or not there
exists a Clifford-like parallelism subject to \eqref{eq:aut-open} remains as an
open problem.

\begin{cor}\label{cor:aut-not-ex}
In\/ $\dspH$, there exists no Clifford-like parallelism\/ $\parallel$
satisfying\/ \eqref{eq:aut-not-ex}.
\end{cor}
\begin{proof}
If \eqref{eq:aut-not-ex} holds for some Clifford-like parallelism $\parallel$,
then, by Theorem~\ref{thm:cl-aut}, there exists an antiautomorphism $\alpha_1$
of $H$ such that $\alpha_1\in\Autp$.
Corollary~\ref{cor:skolem-noether-pasotti}~\eqref{cor:skolem-noether-pasotti.b}
shows $\alpha_1\circ\quer\in\Aut(H)\setminus\Autp$. But \eqref{eq:semidir1} and
\eqref{eq:aut-not-ex} force
$\alpha_1\circ\quer\in\Aut(H)\subset\Autle\subset\Autp$, an absurdity.
\end{proof}

\begin{rem}\label{rem:correl} For any Clifford-like parallelism
$\parallel$ of $\dspH$ there are also correlations that preserve $\parallel$.
We just give one
example. The orthogonality relation $\perp$ that stems from the non-degenerate
symmetric bilinear form \eqref{eq:<,>} determines a projective polarity of
$\bPH$ by sending any subspace $S$ of $H_F$ to its orthogonal space $S^\perp$.
Using \cite[Cor.~4.4]{havl+p+p-19a} or \cite[(2.6)]{kk-75} one obtains that $\leS(M)
\cap \riS(M) = \{ M, M^\perp \}$ for all lines $M\in\cLH$. So, for all
$M\in\cLH$, we have $ M\lep M^\perp$ and $M\rip M^\perp$, which implies
$M\parallel M^\perp$. In other words, the polarity $\perp$ fixes all parallel
classes of the parallelisms $\lep$, $\rip$ and $\parallel$. Consequently, each
of the parallelisms $\lep$, $\rip$ and $\parallel$ is preserved under the
action of $\perp$ on the line set $\cLH$.
\end{rem}

\section{Examples}\label{sect:exa}

We first turn to equation~\eqref{eq:aut-equal}, that is, ${\Autp}={\Autle}$. In
any projective double space $\dspH$, this equation has two trivial solutions,
namely ${\parallel}={\lep}$ and, by \eqref{eq:le=ri}, ${\parallel}={\rip}$.
According to \cite[Thm.~4.12]{havl+p+p-19a}, which relies on \cite{fein+s-76a},
a projective double space $\dspH$ admits no Clifford-like parallelisms other
than $\lep$ and $\rip$ precisely when $F$ is a formally real Pythagorean field
and $H$ is the ordinary quaternion skew field over $F$. (See also
\cite[Thm.~9.1]{blunck+k+s+s-17z}.) Thus, when looking for non-trivial
solutions of \eqref{eq:aut-equal}, we have to avoid this particular class of
quaternion skew fields.

\begin{exa}\label{exa:F-Aut-invar}
Let $H$ be any quaternion skew field of characteristic two. From
\cite[Ex.~4.13]{havl+p+p-19a}, there exists a Clifford-like parallelism
$\parallel$ of $\dspH$ such that $\cF$ comprises \emph{all} lines
$L\in\cA(H_F)$ that are---in an algebraic language---separable extensions of
$F$. The set $\cF$ is fixed under all automorphisms of $H$, since any
$L'\in\cA(H_F)\setminus \cF$ is an inseparable extension of $F$.
Equation~\eqref{eq:semidir1} and Theorem~\ref{thm:cl-aut} together give
${\Autle}\subseteq{\Autp}$. As \eqref{eq:aut-not-ex} cannot apply, we get
${\Autle}={\Autp}$. Each of the sets $\cF$ and $\cA(H_F)\setminus\cF$ is
non-empty; see, for example, \cite[pp.~103--104]{draxl-83} or
\cite[pp.~46--48]{tits+w-02a}. Hence $\parallel$ does not coincide with $\lep$
or $\rip$.
\end{exa}

\begin{exa}\label{exa:aut=inner}
Let $H$ be a quaternion skew field that admits only inner automorphisms. Then
all automorphisms and all antiautomorphisms of $H$ are in $\GL(H_F)$. By
Theorem~\ref{thm:cl-aut-lin}, ${\Autle}$ is the common automorphism group of
all Clifford-like parallelisms of $\dspH$.
\par
In particular, any quaternion skew field $H$ with centre $\bQ$ admits only
inner automorphisms by the Skolem-Noether theorem. Since $\bQ$ is not
Pythagorean, we may infer from \cite[Thm.~4.12]{havl+p+p-19a} that any
$\bigl(\bP(H_\bQ),{\lep},{\rip}\bigr)$ possesses Clifford-like parallelisms
other than $\lep$ and $\rip$. (See \cite[Ex.~4.14]{havl+p+p-19a} for detailed
examples.)
\end{exa}

In order to establish the existence of Clifford-like parallelisms $\parallel$
that satisfy \eqref{eq:aut-proper}, we shall consider certain quaternion skew
fields admitting an outer automorphism of order two. The idea to use this kind
of automorphism stems from the theory of \emph{involutions of the second kind}
\cite[\S2,~2.B.]{knus+m+r+t-98a}. Indeed, for each of the automorphisms
$\alpha$ from Examples~\ref{exa:root3}, \ref{exa:c2-sep}, \ref{exa:c2-sep-old},
\ref{exa:c2-insep} and \ref{exa:c2-insep-old} the product $\alpha\circ\quer$ is
such an involution. Also, we shall use the following auxiliary result.

\begin{lem}\label{lem:orbit}
Let $L$ be a maximal commutative subfield of $H$, let $\alpha\in \Aut(H)$, and
let $h\in H^*$. Furthermore, assume that $\alpha(L) = h^{-1} L h$.
\begin{enumerate}
\item\label{lem:orbit.a} If\/ $\Char H \neq 2$, then for each $q\in
    L\setminus\{0\}$ with\/ $\tr(q)=0$ there exists an element $c\in F^*$
    such that
    \begin{equation}\label{eq:c}
    c^2 = N\bigl(\alpha(q)\bigr) \, N(q)^{-1} . 
    \end{equation}
\item\label{lem:orbit.b} If\/ $\Char H = 2$ and $L$ is separable over $F$,
    then for each $q\in L$ with $\tr(q)=1$ there exists an element $d\in F$
    such that
    \begin{equation}\label{eq:d}
    d^2 +d = N\bigl(\alpha(q)\bigr) + N(q) .
\end{equation}
\item\label{lem:orbit.c} If\/ $\Char H = 2$ and $L$ is inseparable over
    $F$, then for each $q\in L\setminus F$ there exist elements $c\in F^*$,
    $d\in F$ such that
    \begin{equation}\label{eq:e}
    d^2= N\bigl(\alpha(q)\bigr)+c^2N(q).
    \end{equation}
\end{enumerate}
\end{lem}
\begin{proof}
\eqref{lem:orbit.a} From \eqref{eq:tr+N}, applied first to $\alpha$ and then to
the inner automorphism $\tilde{h}$, we obtain
$\tr\bigl(\alpha(q)\bigr)=0=\tr(h^{-1}qh)$. The elements of $\alpha(L)$ with
trace zero constitute a one-dimensional $F$-subspace of $\alpha(L)$. Hence
there exists an element $c\in F^*$ with $\alpha(q)= c (h^{-1}qh)$. Application
of the norm function $N$ establishes \eqref{eq:c}.

\par
\eqref{lem:orbit.b} Like before, \eqref{eq:tr+N} implies
$\tr\bigl(\alpha(q)\bigr)=1=\tr(h^{-1}qh)$. The elements of $\alpha(L)$ with
trace $1$ constitute the set $\alpha(q)+F \subset \alpha(L)$. Hence there
exists an element $d\in F$ with $\alpha(q)+ d = h^{-1}qh$. Taking the norm on
both sides gives $N\bigl(d+\alpha(q)\bigr)=N(q)$. This equation can be
rewritten as in \eqref{eq:d}, which follows from $N\bigl(\alpha(q)+
d\bigr) = \bigl(\alpha(q)+ d\bigr)\bigl(\overline{\alpha(q)+ d}\bigr) =
\bigl(\alpha(q)+ d\bigr)\bigl(\alpha(q)+ d+1\bigr)$.

\par
\eqref{lem:orbit.c} Since both $L$ and $\alpha(L)$ are inseparable over $F$,
for any $x\in L\cup\alpha(L)$ it follows $\tr(x)=0$ and, by \eqref{eq:x1},
$N(x)=x^2$. Thus, in particular, $\tr\bigl(\alpha(q)\bigr)=0=\tr(h^{-1}qh)$.
Since $\alpha(q)$ belongs to $\alpha(L)$, which is a $2$-dimensional $F$-vector
space spanned by $h^{-1}qh$ and $1$, there exist $c,d\in F$ such that
$\alpha(q)=c(h^{-1}qh)+d$. Note that $c\neq0$ since $\alpha(q)\notin F$. Taking
the norm on both sides of the previous equation gives
$N\bigl(\alpha(q)\bigr)=N\bigl(c(h^{-1}qh)+d\bigr) =
\bigl(c(h^{-1}qh)+d\bigr)^2=c^2N(q)+d^2$, which entails \eqref{eq:e}.
\end{proof}

\begin{exa}\label{exa:root3}
Let $F=\bQ\bigl(\sqrt{3}\bigr)$ and denote by $H$ the ordinary quaternions over
$F$ with the usual $F$-basis $\{1,i,j,k\}$. The mapping $v + w \sqrt{3}\mapsto
v - w \sqrt{3}$, $v,w\in \bQ$, is an automorphism of $F$. It can be extended to
a unique $F$-semilinear transformation, say $\alpha\colon H\to H$, such that
$\{1,i,j,k\}$ is fixed elementwise. This $\alpha$ is an automorphism of the
skew field $H$, since all structure constants of $H$ with respect to the given
basis are in $\bQ$, and so all of them are fixed under $\alpha$.
\par
Following Lemma~\ref{lem:orbit}, we define $q:=i+\bigl(1+\sqrt{3}\bigr)j$ and
$L:=F1\oplus Fq$. Then $\tr(q)=q+\overline{q}=0$,
\begin{equation*}
    N(q) = 1 + \bigl(1+\sqrt{3}\bigr)^2 = 5+2\sqrt{3},
    \quad N\bigl(\alpha(q)\bigr)=\alpha\bigl(N(q)\bigr)=5-2\sqrt{3}
\end{equation*}
and $N\bigl(\alpha(q)\bigr)\,N(q)^{-1} =
\bigl(5-2\sqrt{3}\bigr)^2 /13 \neq c^2$ for all $c\in F^*$, since $13$ is not
a square in $F$. By Lemma~\ref{lem:orbit}~\eqref{lem:orbit.a}, there is no
$h\in H^*$ such that $\alpha(L)=h^{-1}Lh$.
\par
We now apply the construction from \cite[Thm.~4.10~(a)]{havl+p+p-19a} to the
set $\cD:=\{L\}$. This gives a Clifford-like parallelism $\parallel$ with the
property $\cF=\{h^{-1}L h\mid h\in H^*\}$. Under the action of the group of
inner automorphisms, $\inner{H^*}$, the star $\cA(H_F)$ splits into orbits of
the form $\{h^{-1}L'h\mid h\in H^*\}$ with $L'\in\cA(H_F)$. One such orbit is
$\cF$ and, due to $\alpha(L)\notin\cF$, another one is $\alpha(\cF)$. The
automorphism $\alpha$ interchanges these two distinct orbits, but it fixes the
$\inner{H^*}$-orbit of the line $F1\oplus Fi$. Therefore,
$\cA(H_F)\setminus\cF$ contains at least two distinct $\inner{H^*}$-orbits.
Consequently, there is no antiautomorphism of $H$ taking $\cF$ to
$\cA(H_F)\setminus\cF$. So, by Theorem~\ref{thm:cl-aut},
$\Autp\subseteq\Autle$. From \eqref{eq:semidir1}, Theorem~\ref{thm:cl-aut} and
$\alpha(L)\notin\cF$, follows $\alpha\in\Autle\setminus\Autp$. Summing up, we
have $\Autp\subset\Autle$, as required.
\end{exa}

\begin{exa}\label{exa:c2-sep}
Let $\bF_2$ be the Galois field with two elements, and let $F=\bF_2(t,u)$,
where $t$ and $u$ denote independent indeterminates over $\bF_2$.
\par
First, we collect some facts about the polynomial algebra $\bF_2[t,u]$ over
$\bF_2$. Let $\bN$ denote the set of non-negative integers.
The monomials of the form
\begin{equation}\label{eq:basis}
    t^\gamma u^\delta \mbox{~~with~~}(\gamma,\delta)\in\bN\times\bN
\end{equation}
constitute a basis of the $\bF_2$-vector space $\bF_2[t,u]$. Each non-zero
polynomial $p\in\bF_2[t,u]$ can be written in a unique way as a non-empty sum
of basis elements from \eqref{eq:basis}. Among the elements in
this sum there is a unique one, say $t^m u^n$, such that $(m,n)$ is maximal
w.r.t.\ the lexicographical order on $\bN\times \bN$. We shall refer to $(m,n)$
as the \emph{$t$-leading pair} of $p$. (In this
definition the indeterminates $t$ and $u$ play different roles, because of the
lexicographical order. Due to this lack of symmetry the degree of $p$ can be
strictly larger than $m+n$.) If $p_1,p_2\in\bF_2[t,u]$ are non-zero polynomials
with $t$-leading pairs $(m_1,n_1)$ and $(m_2,n_2)$, then $p_1p_2$ is
immediately seen to have the $t$-leading pair $(m_1+m_2,n_1+n_2)$.
\par
Next, we construct a quaternion algebra with centre $F$. We follow the
notation from \cite{blunck+p+p-10a} and
\cite[Rem.~3.1]{havl-16a}. Let $K:=F(i)$ be a separable quadratic extension of
$F$ with defining relation $i^2+i+1=0$. Furthermore, we define $b:=t+u$. The
quaternion algebra $(K/F,b)$ has a basis $\{1,i,j,k\}$ such that its
multiplication is given by the following table:
\begin{equation*}\label{eq:table}\small
    \begin{array}{c|ccc}
    \cdot &i  & j     & k \\
    \hline
    i     &1+i& k     & j+k \\
    j     &j+k& t+u   & (t+u)(1+i) \\
    k     &j  & (t+u)i& t+u
    \end{array}
\end{equation*}
The conjugation $\quer\colon H\to H$ sends $i\mapsto\ol i = i+1$ and fixes both
$j$ and $k$.
\par
In order to show that $(K/F,b)$ is a skew field we have to verify $b\notin
N(K)$. Assume to the contrary that there are polynomials $p_1$, $p_2\neq 0$,
$p_3$, and $p_4\neq 0$ in $\bF_2[t,u]$ such that
\begin{equation*}
\begin{aligned}
    N\big( p_1 / p_2 + (p_3 /p_4) i\big) &= \big( p_1 /p_2 + (p_3 /p_4) i\big)
    \big( p_1 / p_2 + (p_3 /p_4)(i+1)\big)\\
      &= p_1^2 / p_2^2 + (p_1 p_3) / (p_2 p_4) + p_3^2 / p_4^2\\
      &= t + u .
\end{aligned}
\end{equation*}
Consequently,
\begin{equation}\label{eq:cond}
      ( p_1  p_4 )^2 +  p_1 p_2 p_3 p_4  + ( p_2p_3 )^2
      + (t + u)(p_2p_4)^2 = 0.
\end{equation}
We cannot have $p_1=0$ or $p_3=0$, since then the left hand side of
\eqref{eq:cond} would reduce to a sum of two terms, with one being a square in
$\bF_2[t,u]$ and the other being a non-square. We define $(m_s,n_s)$ to be the
$t$-leading pair of $p_s$, $s\in\{1,2,3,4\}$. So, the $t$-leading pairs of the
first three summands on the left hand side of \eqref{eq:cond} are
\begin{equation*}\label{eq:t-leaders}
\begin{aligned}
    &\bigl(2(m_1+m_4),2(n_1+n_4)\bigr),\; (m_1+m_2+m_3+m_4,n_1+n_2+n_3+n_4),\;\\
    &\bigl(2(m_2+m_3),2(n_2+n_3)\bigr).
\end{aligned}
\end{equation*}
Let us expand each of the four summands on the left hand side of
\eqref{eq:cond} in terms of the monomial basis \eqref{eq:basis}. All monomials
in the fourth expansion have odd degree. There are three possibilities.
\par
\emph{Case}~(i). $m_1+m_4\neq m_2+m_3$. Then, for example, $m_1+m_4 >m_2+m_3$.
From
\begin{equation}\label{eq:greater}
    2(m_1+m_4)>m_1+m_2+m_3+m_4 > 2(m_2+m_3) ,
\end{equation}
the monomial $t^{2(m_1+m_4)}u^{2(n_1+n_4)}$ appears in the expansion of
$(p_1p_4)^2$, but not in the expansions of $p_1p_2p_3p_4$ and $(p_2p_3)^2$.
This monomial remains unused in the expansion of $(t + u)(p_2p_4)^2$, since
both of its exponents are even numbers. So, the left hand side of
\eqref{eq:cond} does not vanish, whence this case cannot occur.
\par
\emph{Case}~(ii). $m_1+m_4 = m_2+m_3$ and $n_1+n_4\neq n_2+n_3$. Then, for
example, $n_1+n_4>n_2+n_3$. Formula~\eqref{eq:greater} remains true when
replacing $m_s$ by $n_s$, $s\in\{1,2,3,4\}$. We now can deduce, as in Case~(i),
that the monomial $t^{2(m_1+m_4)}u^{2(n_1+n_4)}$ appears precisely once when
expanding each of the four summands the left hand side of \eqref{eq:cond} in
the monomial basis. So, this case is impossible.
\par
\emph{Case}~(iii). $m_1+m_4 = m_2+m_3$ and $n_1+n_4 = n_2+n_3$. Then
$t^{2(m_1+m_4)}u^{2(n_1+n_4)}$ appears precisely three times when
expanding the four summands on the left hand side of \eqref{eq:cond} and, due
to $1+1+1\neq 0$, this case cannot happen either.
\par
Since none of the Cases~(i)--(iii) applies, we end up with a contradiction.
\par
There is a unique automorphism of $F$ that interchanges the indeterminates $t$
and $u$. It can be extended to a unique $F$-semilinear transformation, say
$\alpha\colon H\to H$, such that $\{1,i,j,k\}$ is fixed elementwise. This
$\alpha$ is an automorphism of $H$, because $\alpha(t+u)=u+t=t+u$.
\par
Following Lemma~\ref{lem:orbit}, we define $q:=i+ u j$ and $L:=F1\oplus Fq$.
Then $\tr(q)=1$, $N(q)=1 +u^2(t+u)$, and $N\bigl(\alpha(q)\bigr)=1+t^2(u+t)$.
\par
We claim that
\begin{equation*}
    N\bigl(\alpha(q)\bigr) + N(q) = (u+t)^3 \neq d^2+d
    \mbox{~~for all~~}d\in F.
\end{equation*}
Let us assume, by way of contraction, that there are polynomials $d_1$ and
$d_2\neq 0$ in $\bF_2[t,u]$ satisfying $(u+t)^3 = d_1^2 / d_2^2 + d_1 /
d_2$.
Hence $d_1\neq 0$ and
\begin{equation}\label{eq:poly=0}
    (u+t)^3 d_2^2 + d_1^2 + d_1d_2 = 0 .
\end{equation}
We expand the first summand in \eqref{eq:poly=0} in terms of the monomial basis
\eqref{eq:basis}. This gives a sum of monomials all of which have odd degree.
Likewise, the expansion of the second summand in \eqref{eq:poly=0} results in a
sum of monomials all of which have even degree. Let us also expand the third
summand in \eqref{eq:poly=0} to a sum of monomials and let us then collect all
monomials with odd (resp.\ even) degree. In this way we get precisely the
monomials appearing in the first (resp.\ second) sum from above. Thus, with
$n_1:=\deg d_1$, $n_2:=\deg d_2$ we obtain that the degrees of the
summands in \eqref{eq:poly=0} satisfy the inequalities
\begin{equation*}
     3+2 n_2 \leq  n_1+n_2 ,\; 2 n_1\leq n_1+n_2.
\end{equation*}
These inequalities imply $3+2n_2\leq n_1+n_2\leq n_2+n_2$, which is absurd.
\par
By Lemma~\ref{lem:orbit}~\eqref{lem:orbit.b}, there is no $h\in H^*$ such that
$\alpha(L)=h^{-1}Lh$.
\par
We now repeat the reasoning from the end of
Example~\ref{exa:root3}. This shows that the Clifford-like parallelism
$\parallel$ that arises from $\cD:=\{L\}$ satisfies $\Autp\subset\Autle$.
\end{exa}

\begin{exa}\label{exa:c2-sep-old}
Let $H=(K/F,b)$, $\alpha\in\Aut(H)$ and $L$ be given as in
Example~\ref{exa:c2-sep}. We know from
Example~\ref{exa:F-Aut-invar} that
\begin{equation}\label{eq:E-insep}
    \cE_{\mathrm{insep}}:=\bigl\{L'\in\cA(H_F)\mid L'/F \mbox{~is~inseparable}\bigr\} \neq \emptyset.
\end{equation}
In contrast to Example~\ref{exa:c2-sep}, we adopt an alternative definition of
$\cD$, namely $\cD:=\{L\}\cup \cE_{\mathrm{insep}}$. The construction from
\cite[Thm.~4.10~(a)]{havl+p+p-19a} applied to this $\cD$ gives a Clifford-like
parallelism $\parallel$ with the property $\cF=\{h^{-1}L h\mid h\in
H^*\}\cup\cE_{\mathrm{insep}}$. The set $\cE_{\mathrm{insep}}$ remains fixed
under any antiautomorphism of $H$. Consequently, there is no antiautomorphism
of $H$ taking $\cF$ to $\cA(H_F)\setminus\cF$. So, by Theorem~\ref{thm:cl-aut},
$\Autp\subseteq\Autle$. From \eqref{eq:semidir1}, Theorem~\ref{thm:cl-aut} and
$\alpha(L)\notin\cF$, follows $\alpha\in\Autle\setminus\Autp$. Summing up, we
have $\Autp\subset\Autle$, as required.
\end{exa}

\begin{exa}\label{exa:c2-insep}
Consider the same quaternion skew field $H=(K/F,b)$ and the same automorphism
$\alpha\in\Aut(H)$ as in Example~\ref{exa:c2-sep}. However, now we define
$q:=j+uk$ and $L:=F1\oplus Fq$. Then $L$ is inseparable over $F$, $\tr(q)=0$,
$N(q)=(j+uk)^2=(u+t)(1+u+u^2)$ and $N\bigl(\alpha(q)\bigr)=(u+t)(1+t+t^2)$.
Equation~\eqref{eq:e} of Lemma~\ref{lem:orbit} is
\begin{equation*}
(u+t)(1+t+t^2)+c^2(u+t)(1+u+u^2)=d^2
\end{equation*}
which, upon fixing $c={c_1}/{c_2}$ and $d={d_1}/{d_2}$ with $c_1,c_2,d_1,d_2\in
\bF_2[t,u]$ and $c_1,c_2,d_2\neq 0$, is equivalent to
\begin{equation}\label{eq:insep1}
    d_2^2(u+t)(c_2^2+c_2^2t^2+c_1^2+c_1^2u^2)+d_2^2(u+t)(c_2^2t+uc_1^2)=d_1^2c_2^2.
\end{equation}
All the monomials in the first summand of the left
hand side of equation~\eqref{eq:insep1} are of odd degree, while the monomials
in the second one are of even degree. Since $d_1^2 c_2^2$ is a sum of monomials
of even degree this entails
\begin{equation}\label{eq:insep2}\renewcommand{\arraystretch}{1.2}
    \left\{\begin{array}{l}
        d_2^2(u+t)(c_2^2+c_2^2t^2+c_1^2+c_1^2u^2)=0, \\
        d_2^2(u+t)(c_2^2t+uc_1^2)=d_1^2c_2^2.
    \end{array}\right.
\end{equation}
The second equation in \eqref{eq:insep2} yields $ut(1+c)^2=(d+t+uc)^2$, and
since $c=1$ (\emph{i.e.}, $c_1=c_2$) is not a solution of the first equation in
\eqref{eq:insep2}, we can assume $1+c\neq0$, thus
$ut=\big((d+t+uc)/(1+c)\big)^2$.
This equation, after all, cannot be satisfied for any choice of $c\in F^*$
since $ut$ is not a square in $F$. Thus we can conclude by Lemma
\ref{lem:orbit}~\eqref{lem:orbit.c} that there exists no $h\in H^*$ such that
$\alpha(L)=h^{-1}qh$.
\par
The final step is to define a Clifford-like parallelism subject to
\eqref{eq:aut-proper}. This can be done as in Example~\ref{exa:c2-sep} using
$\cD:=\{L\}$.
\end{exa}

\begin{exa}\label{exa:c2-insep-old}
Let $H=(K/F,b)$, $\alpha\in\Aut(H)$ and $L$ be given as in
Example~\ref{exa:c2-insep}. Then a Clifford-like parallelism that satisfies
\eqref{eq:aut-proper} can be obtained along the lines of
Example~\ref{exa:c2-sep-old} by replacing everywhere the set
$\cE_{\mathrm{insep}}$ from \eqref{eq:E-insep} with $\cE_{\mathrm{sep}}:=
\bigl\{L'\in\cA(H_F)\mid L'/F \mbox{~is~separable}\bigr\}$.
\end{exa}


\providecommand{\bysame}{\leavevmode\hbox to3em{\hrulefill}\thinspace}
\providecommand{\MR}{\relax\ifhmode\unskip\space\fi MR }
\providecommand{\MRhref}[2]{%
  \href{http://www.ams.org/mathscinet-getitem?mr=#1}{#2}
}
\providecommand{\href}[2]{#2}

\noindent
Hans Havlicek\\
Institut f\"{u}r Diskrete Mathematik und Geometrie\\
Technische Universit\"{a}t\\
Wiedner Hauptstra{\ss}e 8--10/104\\
A-1040 Wien\\
Austria\\
\texttt{havlicek@geometrie.tuwien.ac.at}
\par~\par
\noindent Stefano Pasotti\\
DICATAM-Sez. Matematica\\
Universit\`{a} degli Studi di Brescia\\
via Branze, 43\\
I-25123 Brescia\\
Italy\\
\texttt{stefano.pasotti@unibs.it}
\par~\par
\noindent
Silvia Pianta\\
Dipartimento di Matematica e Fisica\\
Universit\`{a} Cattolica del Sacro Cuore\\
via Trieste, 17\\
I-25121 Brescia\\
Italy\\
\texttt{silvia.pianta@unicatt.it}

\end{document}